\documentclass[11pt]{amsart}
\usepackage{graphicx}
\usepackage{amssymb, amsmath}
\newtheorem{theorem}{Theorem}[section]
\newtheorem{corollary}[theorem]{Corollary}

\theoremstyle{definition}

\theoremstyle{remark}

\numberwithin{equation}{section}

\newcommand{\Gf}{\mathrm{der}_{\theta, b}(G, \cdot)}

\begin{document}
\title{Equations in  polyadic groups}
\author{H. Khodabandeh}
\address{H. Khodabandeh\\
Department of Pure Mathematics,  Faculty of Mathematical
Sciences, University of Tabriz, Tabriz, Iran }
\author{\sc M. Shahryari}
\thanks{{\scriptsize
\hskip -0.4 true cm MSC(2010): Primary 20N15, Secondary 08A99 and 14A99
}}

\address{M. Shahryari\\
 Department of Pure Mathematics,  Faculty of Mathematical
Sciences, University of Tabriz, Tabriz, Iran }

\email{mshahryari@tabrizu.ac.ir}
\date{\today}

\begin{abstract}
Systems of equations and their solution sets are studied in polyadic groups. We prove that a polyadic group $(G, f)=\mathrm{der}_{\theta, b}(G, \cdot)$ is equational noetherian, if and only if the ordinary group $(G, \cdot)$ is equational noetherian. The structure of coordinate polyadic group of algebraic sets in equational noetherian polyadic groups are also determined.
\end{abstract}

\maketitle


\section{Introduction}
In the recent years, the study of equations in various algebraic structures has been  the subject of many research activities. Investigation of equations over groups, especially free  and hyperbolic groups, was the main motivation for the invention of algebraic geometry over groups, \cite{BMR1}. This interesting subject is the result of decades of efforts of many algebraists to solve the classical problems of Alfred Tarski concerning the first order theory of non-abelian free groups (see \cite{KMTarski}, \cite{KMTarski2} and \cite{KM}). Equations over Lie algebras, associative algebras and semigroups are also studied  (for example, see \cite{DMR2}, \cite{DR1}, and \cite{SHEV}). In a series of fundamental papers, Daniyarova, Myasnikov and Remeslennikov introduced the basics of {\em universal algebraic geometry} (\cite{DMR1}, \cite{DMR2}, \cite{DMR3} and \cite{DMR4}).

In this article, we study equations over {\em polyadic groups}. A polyadic group is a  natural generalization of the concept of group to the case where the binary operation of group replaced with an $n$-ary associative operation, one variable linear equations in which have unique solutions (see the next section for the detailed definitions). These interesting algebraic objects are introduced by Kasner and D\"ornte (\cite{Kas} and \cite{Dor}) and  studied extensively by Emil Post during the first decades of the last century, \cite{Post}. During decades, many articles are published on the structure of polyadic groups. The authors of this article, already studied homomorphisms and automorphisms of polyadic groups, \cite{Khod-Shah}. A characterization of the simple polyadic groups is obtained by them in \cite{Khod-Shah2}. The second author with W. Dudek, studied representation theory of polyadic groups, \cite{Dud-Shah}. The complex characters of finite polyadic groups are also investigated by the second author in \cite{Shah2}.

We will begin to study of algebraic geometry over polyadic groups. It is known that for every polyadic ($n$-ary) group $(G, f)$, there exists a corresponding ordinary group $(G, \cdot)$, an automorphism $\theta$ of this ordinary group, and an element $b\in G$, the structure of $(G, f)$ in which complectly determined by  $(G, \cdot)$, $\theta$, and $b$. Our main result shows that the algebraic geometries of the polyadic group $(G, f)$ and the corresponding group $(G, \cdot)$ are essentially the same. This algebraic geometry does not depend on the automorphism $\theta$ and the special element $b$. This result gives us an interesting application of polyadic groups for the algebraic geometry over groups. Here we describe this application briefly.  Let $G$ be an ordinary group and suppose $\theta$ is an automorphism of $G$ satisfying the following two conditions:\\

{\em 1- There exists a fixed point $b\in G$ for $\theta$.\\

2- There exists a natural number $n\geq 2$ such that for all $x\in G$, $\theta^{n-1}(x)=bxb^{-1}$}.\\

Let $\mathcal{L}$ be the ordinary language of groups containing elements of $G$ as parameters. We obtain an extended language $\mathcal{L}_{\theta}$ by adding $\theta$ to $\mathcal{L}$ as a unary functional symbol. We will show that if the group $G$ is equational noetherian as an $\mathcal{L}$-algebra, then it is also equational noetherian as an $\mathcal{L}_{\theta}$-algebra (similar statements are also true for the properties of being $q_{\omega}$-compact and $u_{\omega}$-compact).

A compact review of the basic concepts of polyadic group is given in the next section. Our notations dealing with universal algebraic geometry, are the same as \cite{DMR1}, \cite{DMR2}, \cite{DMR3}, \cite{DMR4}. The reader should consult  these references for a complete account of the universal algebraic geometry.

\section{Polyadic groups}
This section contains basic notions and properties of polyadic groups as well as some literature.
Let $G$ be a non-empty set and $n$ be a positive integer. If
$f:G^n\to G$ is an $n$-ary operation, then we use the compact
notation $f(x_1^n)$ for the elements $f(x_1, \ldots, x_n)$. In
general, if $x_i, x_{i+1}, \ldots, x_j$ is an arbitrary sequence of
elements in $G$, then we denote it as $x_i^j$. In the special case,
when all terms of this sequence are equal to a constant $x$, we denote
it by $\stackrel{(t)}{x}$, where $t$ is the number of terms.  We say that an $n$-ary
operation is {\em associative}, if for any $1\leq i<j\leq n$, the
equality
$$
f(x_1^{i-1},f(x_i^{n+i-1}),x_{n+i}^{2n-1})=
f(x_1^{j-1},f(x_j^{n+j-1}),x_{n+j}^{2n-1})
$$
holds for all $x_1,\ldots,x_{2n-1}\in G$. An $n$-ary system $(G,f)$
is called an {\em $n$-ary group} or a {\em polyadic group}, if $f$
is associative and for all $a_1,\ldots,a_n, b\in G$ and $1\leq i\leq
n$, there exists a unique element $x\in G$ such that
$$
f(a_1^{i-1},x,a_{i+1}^n)=b.
$$
It is proved that the uniqueness assumption on the solution $x$ can
be dropped \cite{Dud2}. Clearly, the case $n=2$ is
just the definition of ordinary groups. During
this article, we assume that $n> 2$.
T
he classical paper of E. Post \cite{Post},  is
one of the first articles published on the subject. In this paper,
Post proves his well-known {\em coset theorem}. Many basic
properties of polyadic groups are studied in this paper. The
articles \cite{Dor} and \cite{Kas} are among the first materials
written on the polyadic groups. Russian reader, can use
the book of Galmak \cite{Gal}, for an almost complete description of
polyadic groups.
The articles \cite{Art}, \cite{Dud1}, \cite{Gal2}, and \cite{Gle}
can be used for study of axioms of polyadic groups as well as their
varieties.

Note that an $n$-ary system $(G,f)$ of the form $\,f(x_1^n)=x_1
x_2\ldots x_nb$, where $(G,\cdot)$ is a group and $b$  a fixed
element belonging to the center of $(G,\cdot)$, is an $n$-ary group.
Such an $n$-ary group is called {\em $b$-derived} from the group
$(G,\cdot)$ and it is denoted by $\mathrm{der}_b^n(G, \cdot)$. In the case when $b$ is the identity of $(G,\cdot)$, we
say that such a polyadic group is {\em reduced} to the group $(G, \cdot
)$ or {\em derived} from $(G, \cdot)$ and we use the notation $\mathrm{der}^n(G, \cdot)$ for it. For every $n>2$, there
are $n$-ary groups which are not derived from any group. An $n$-ary
group $(G,f)$ is derived from some group if and only if it contains
an element $a$ (called an {\em $n$-ary identity}) such that
$$
 f(\stackrel{(i-1)}{a},x,\stackrel{(n-i)}{a})=x
$$
holds for all $x\in G$ and for all $i=1,\ldots,n$, see \cite{Dud2}.

From the definition of an $n$-ary group $(G,f)$, we can directly see
that for every $x\in G$, there exists only one $y\in G$,  satisfying
the equation
$$
f(\stackrel{(n-1)}{x},y)=x.
$$
This element is called {\em skew} to $x$ and it is denoted by
$\overline{x}$.    As D\"ornte \cite{Dor} proved, the
following identities hold for all $\,x,y\in G$, $2\leq i\leq n$,
$$
f(\stackrel{(i-2)}{x},\overline{x},\stackrel{(n-i)}{x},y)=
f(y,\stackrel{(n-i)}{x},\overline{x},\stackrel{(i-2)}{x})=y.
$$
These identities together with the associativity identities, axiomatize the variety of polyadic groups in the algebraic language $(f, ^{-})$.

Suppose $(G, f)$ is a polyadic group and $a\in G$ is a fixed
element. Define a binary operation
$$
x\ast y=f(x,\stackrel{(n-2)}{a},y).
$$
Then $(G, \ast)$ is an ordinary group, called the {\em retract} of
$(G, f)$ over $a$. Such a retract will be denoted by $ret_a(G,f)$. All
retracts of a polyadic group are isomorphic \cite{DM}. The
identity of the group $(G,\ast)$ is $\overline{a}$. One can verify
that the inverse element to $x$ has the form
$$
y=f(\overline{a},\stackrel{(n-3)}{x},\overline{x},\overline{a}).
$$

One of the most fundamental theorems of polyadic group is the
following, now known as {\em Hossz\'{u} -Gloskin's Theorem}. We will use
it frequently in this article and the reader can use \cite{DG},
\cite{DM1}, \cite{Hos} and \cite{Sok} for detailed discussions.

\begin{theorem}
Let $(G,f)$ be an $n$-ary group. Then there exists an ordinary group $(G, \cdot)$,
an automorphism $\theta$ of $(G, \cdot)$ and an element $b\in G$ such that

1.  $\theta(b)=b$,

2.  $\theta^{n-1}(x)=b x b^{-1}$, for every $x\in
G$,

3.  $f(x_1^n)=x_1\theta(x_2)\theta^2(x_3)\cdots\theta^{n-1}(x_n)b$, for
all $x_1,\ldots,x_n\in G$.

\end{theorem}

According to this theorem, we  use the notation $\Gf$ for $(G,f)$
and we say that $(G,f)$ is $(\theta, b)$-derived from the group $(G,
\cdot)$. During this paper, we will assume that $(G, f)=\Gf$.

Varieties of polyadic groups and the structure of congruences on
polyadic groups are studied in \cite{Art} and \cite{Dud1}. It is proved that all congruences on polyadic groups
are commute and so the lattice of congruences is modular.

We established some  fundamental results on the structure of
homomorphisms of polyadic groups in \cite{Khod-Shah}. The main result of \cite{Khod-Shah} will be very
strong tool to determine the structure of coordinate polyadic groups in the next sections.

\begin{theorem}
Suppose $(G,f)=\mathrm{der}_{\theta, b}(G, \cdot)$ and $(H, h)=\mathrm{der}_{\eta,
c}(H, \ast)$ are two polyadic groups. Let $\psi: (G,f)\to (H, h)$ be
a homomorphism. Then there exists $a\in H$ and an ordinary
homomorphism $\phi:(G, \cdot)\to (H, \ast)$, such that $\psi=R_a\phi$, where $R_a$ denotes the map $x\mapsto x\ast a$. Further $a$ and $\phi$ satisfy the following conditions;\\
$$
h(\stackrel{(n)}{a})=\phi(b)\ast a\ \ \ and \ \ \
\phi\theta=I_{a}\eta\phi,
$$
where, $I_a$ denotes the inner automorphism $x\mapsto a\ast x\ast
a^{-1}$.

Conversely, if $a$ and $\phi$ satisfy the above two conditions, then
$\psi=R_a\phi$ is a homomorphism $(G,f)\to  (H,h)$.
\end{theorem}

There is one more important object associated to polyadic groups. Let $(G, f)$ be a polyadic group. Then, as Post proved, there exists a unique group $(G^{\ast}, \circ)$ (which we call now the Post's cover of $(G, f)$) such that \\

{\em 1- $G$ is contained in $G^{\ast}$ as a coset of a normal subgroup $R$.\\

2- $R$ is isomorphic to a retract of $(G, f)$.\\

3- We have $G^{\ast}/R\cong \mathbb{Z}_{n-1}$.\\

4- Inside $G^{\ast}$, for all $x_1, \ldots, x_n\in G$, we have
$f(x_1^n)=x_1\circ x_2\circ \cdots \circ x_n$. \\

5- $G^{\ast}$ is generated by $G$.}\\

The group $G^{\ast}$ is also universal in the class of all groups having properties 1, 4. More precisely, if $\beta: (G, f)\to \mathrm{der}^n(H, \ast)$ is a polyadic homomorphism, then there exists a unique ordinary homomorphism $h: G^{\ast}\to H$ such that $h_{|_{G}}=\beta$. This universal property characterizes $G^{\ast}$ uniquely. The explicit construction of the Post's cover can be find in \cite{Shah2}.

\section{Free polyadic groups}
To be prepared to study of equations, we need to introduce the free polyadic groups. This section also can be used as an independent survey of free polyadic groups. In \cite{Artam}, Artamonov showed that $(G, f)$ is free, if and only if, its Post cover is free and so he obtained a description of polyadic subgroups of $(G, f)$. Here, we introduce two methods of explicit construction of free polyadic groups. Let $X$
be our alphabet set. For any group word $w=x_1^{\epsilon_1}\cdots x_m^{\epsilon_m}$, we define the hight $ht(w)=\sum_i\epsilon_i$. Clearly this number does not depend on the special representation of $w$. Let $F_{pol}^n(X)$ be the set of all elements $w$ in the ordinary free group $F(X)$ with the property $ht(w)\equiv 1 (\mathrm{mod}\ n-1)$. Define an $n$-ary operation $f$ on this set by
$$
f(w_1, \ldots, w_n)=w_1w_2\cdots w_n,
$$
where the product in the right hand side is the operation of the free group.

\begin{theorem}
$F_{pol}^n(X)$ is the free $n$-ary group on the set $X$ and its Post cover is $F(X)$.
\end{theorem}

\begin{proof}
As $F_{pol}^n(X)$ is closed under $f$, it is clearly a polyadic group. Note that, in this polyadic group we have $\overline{w}=w^{2-n}$.  We show that the Post cover is $F(X)$. To do this, we will use the universal characterization of the Post cover. Let $(H, \ast)$ be an ordinary group and $\beta: F_{pol}^n(X)\to \mathrm{der}^n(H, \ast)$ be a polyadic homomorphism. So, for all $w_1, \ldots w_n$, we have
$$
\beta(f(w_1, \ldots, w_n))=\beta(w_1)\ast \cdots \ast \beta(w_n).
$$
As $X\subseteq F_{pol}^n(X)$, we can restrict $\beta $ to $X$ to obtain  a map $\alpha:X\to H$. Since $F(X)$ is freely generated by $X$, there is a group homomorphism $h:F(X)\to H$ such that $h_{|_X}=\alpha$. We have
\begin{eqnarray*}
h(x_1^{\epsilon_1}\cdots x_m^{\epsilon_m})&=& \alpha(x_1)^{\epsilon_1}\ast \cdots\ast \alpha(x_m)^{\epsilon_m}\\
                                                  &=&\beta(x_1)^{\epsilon_1}\ast \cdots\ast \beta(x_m)^{\epsilon_m}.
\end{eqnarray*}
We prove that the restriction of $h$ to $F_{pol}^n(X)$ is equal to $\beta$. Note that if $w\in F_{pol}^n(X)$, then
$$
\beta(w^{2-n})=\beta(\overline{w})=\overline{\beta(w)}=\beta(w)^{2-n},
$$
because in $\mathrm{der}^n(H, \ast)$, we have $\overline{a}=a^{2-n}$. This shows that, for any positive integer $k$
$$
\beta(w)^k=(\beta(w^{2-n})\ast \beta(w)^{n-3})^k.
$$
Let $w=x_1^{\epsilon_1}\cdots x_m^{\epsilon_m}\in F_{pol}^n(X)$. For any negative exponent $\epsilon_i$, we replace $\beta(x_i)^{\epsilon_i}$ by
$$
(\beta(w^{2-n})\ast \beta(w)^{n-3})^{|\epsilon_i|}.
$$
By the fact $\sum \epsilon_i\equiv 1 (\mathrm{mod} n-1)$, the right hand side of the equality
$$
h(x_1^{\epsilon_1}\cdots x_m^{\epsilon_m})=\beta(x_1)^{\epsilon_1}\ast \cdots\ast \beta(x_m)^{\epsilon_m}
$$
becomes a product of factors of the form $\beta(w_i)$, with $w_i\in F_{pol}^n(X)$ and the number of these factors has the form $k(n-1)+1$. Since $\beta$ is a polyadic homomorphism, by this trick, we obtain the equality
$$
\beta(x_1)^{\epsilon_1}\ast \cdots\ast \beta(x_m)^{\epsilon_m}=\beta(w).
$$
To explain this process, we give an example: let $n=4$ and $X=\{ x, y\}$. Then the element $w=x^2y^{-1}xy^2$ belongs to $F_{pol}^4(X)$. We have
$$
h(w)=\beta(x)^2\ast \beta(y)^{-1}\ast \beta(x)\ast \beta(y)\ast \beta(y).
$$
If we change $\beta(y)^{-1}$ by $\beta(y^{-2})\ast \beta(y)$, we obtain
$$
h(w)=\beta(x)\ast \beta(x)\ast \beta(y^{-2})\ast \beta(y)\ast \beta(x)\ast \beta(y)\ast \beta(y).
$$
This product has $7=2\times 3+1$ factors and $x, y, y^{-2}\in F_{pol}^4(X)$. So, we have
$$
h(w)=\beta(xxy^{-2}yxyy)=\beta(w).
$$
This argument shows that the Post's cover of $F_{pol}^n(X)$ is $F(X)$, which is a free group, hence by  a result of Artamonov (see \cite{Artam}), we see that $F_{pol}^n(X)$ is also a free polyadic group over $X$. Note that, using a similar argument as above, one can directly prove the freeness of $F_{pol}^n(X)$, without the result of Artamonov.

\end{proof}

One may ask how can we express this free polyadic group $F_{pol}^n(X)$ as in the Hossz\'{u}-Gloskin's theorem. We know that
$$
F_{pol}^n(X)=\mathrm{der}_{\theta, b}(E, \cdot),
$$
for some group $(E, \cdot)$, suitable automorphism $\theta$ and some element $b$. Using Sokolov's construction (see \cite{Sok}), we can describe these objects. We know that the group $(E, \cdot)$ is in fact the retract of $F_{pol}^n(X)$ over some element $a$, i.e.
$$
E=ret_a(F_{pol}^n(X)).
$$
Since the retract is a subgroup of the Post's cover, and the Post's cover is free in this case, so $E$ is free (its rank is $(n-1)(|X|-1)+1$). By \cite{Sok}, we can see that $\theta$ is the restriction of the inner automorphism of $F(X)$ corresponding to $a^{2-n}$ to $E$. Also $b=a^{(n-1)(2-n)}$.

There is another way of description of the free polyadic group in the language $(f, ^-)$. Let $X$ be an arbitrary set and let
$$
\overline{X}=\{\overline{x}: x\in X\}
$$
be a set which is disjoint from $X$ and it is in one to one correspondence with $X$. Let $M(X\cup \overline{X})$ be the free monoid over $X\cup \overline{X}$. For $w\in M(X\cup \overline{X})$, the length of $w$ is denoted by $l(w)$. We define a new set
$$
M_p^n(X)=\{ w\in M(X\cup \overline{X}): l(w)\equiv 1(\mathrm{mod} n-1)\}.
$$
The {\it cancelation operations} on $M_p^n(X)$ consists of deleting or inserting of any piece of the form $\stackrel{(i)}{x}\overline{x}\stackrel{(n-i-1)}{x}$, where $0\leq i\leq n-1$. We say that two elements $w_1$ and $w_2$ are equivalent, if they can be transformed to each other by a finite number of cancelations. This is an equivalence relation over $M_p^n(X)$ and so we can look at its quotient set $F_p^n(X)$. Define an $n$-ary operation on $F_p^n(X)$ by
$$
f(w_1, \ldots, w_n)=w_1w_2\cdots w_n.
$$
It is not hard to see that $F_p^n(X)$ is the free polyadic group over $X$ and one can use this new description in the study of equations over polyadic groups.

\section{Presentation of the Post's cover}
The first description of the free polyadic group, enables us to give a nice connection between presentation of a polyadic group and its Post's cover. Let $u, v\in F_{pol}^n(X)$ be two polyadic words. In this case $u\approx v$ is a polyadic relation. But, since $F_{pol}^n(X)\subseteq F(X)$, so we can look at $u\approx v$ (or $uv^{-1}\approx 1$) as a group relation in the same time. In \cite{Artam}, Artamonov proved that of $(G, f)$ is a free polyadic group, then the Post's cover $(G^{\ast}, \circ)$ is a free group. This is equivalent to saying that if $(G, f)$ has the polyadic presentation $\langle X|\emptyset\rangle_{pol}$, then $(G^{\ast}, \circ)$ has presentation $\langle X|\emptyset\rangle_{gr}$. In the following theorem, we generalize this result of Artamonov.

\begin{theorem}
Let $\langle X|R\rangle_{pol}$ be a presentation for the polyadic group $(G, f)$. Then $\langle X|R\rangle_{gr}$ is a presentation for the Post's cover $(G^{\ast}, \circ)$.
\end{theorem}

\begin{proof}
We have
$$
(G, f)=\frac{F_{pol}^n(X)}{[R]},
$$
where $[R]$ is the congruence generated by $R$. Let $R_0=\{ uv^{-1}; (u, v)\in R\}$ and define $M=F(X)/K$, where $K$ is the normal closure of $R_0$ in $F(X)$. Note that we have
$$
M=\langle X|R_0\rangle_{gr}=\langle X|R_0\rangle_{gr}.
$$
We prove that $M$ is the Post's cover of $(G, f)$. Let $\pi:F_{pol}^n(X)\to (G, f)$ and $\pi^{\prime}:F(X)\to M$ be canonical maps. Note that, the correspondence between $R$ and $R_0$ implies $\pi^{\prime}_{|_{F_{pol}^n(X)}}=\pi$. Let $(H, \ast)$ be an arbitrary group and
$$
\beta: (G, f)\to \mathrm{der}^n(H, \ast)
$$
be a polyadic homomorphism. We prove that there exists an ordinary homomorphism $h:M\to H$, such that $h_{|_G}=\beta$. We know that
$$
\beta\pi: F_{pol}^n(X)\to \mathrm{der}^n(H, \ast)
$$
is a polyadic homomorphism. Since $F(X)$ is the Post's cover of $F_{pol}^n(X)$, so there is a homomorphism $h_0:F(X)\to H$, such that
$$
(h_0)_{|_{F_{pol}^n(X)}}=\beta\pi.
$$
We define a map $h:M\to H$ by the rule $h(\pi^{\prime}(w))=h_0(w)$. To prove that this map is well-defined, assume that $\pi^{\prime}(w_1)=\pi^{\prime}(w_2)$. We show that $h_0(w_1)=h_0(w_2)$. Recall that there is a normal subgroup $N$ in $F(X)$ such that $F_{pol}^n(X)=xN$, for an arbitrary $x\in X$ and
$$
F(X)=\bigcup_{i=0}^{n-2}x^iN.
$$
Hence, we have also
$$
F(X)=\bigcup_{i=0}^{n-2}x^{i-1}F_{pol}^n(X).
$$
Therefore, we have $w_1=x^{i-1}w^{\prime}_1$ and $w_2=x^{j-1}w^{\prime}_2$, for some $0\leq i\leq j\leq n-2$ and $w^{\prime}_1, w^{\prime}_2\in F_{pol}^n(X)$. First, we prove that $i=j$.  For $v\in K$, we have
$$
v=\prod_k a_ku_kv^{-1}_ka^{-1}_k,
$$
where $(u_k, v_k)\in R$ and $a_k\in F(X)$. Hence, we must have 
$$
ht(v)\equiv 0(\mathrm{mod}\ n-1).
$$
Now since, $\pi^{\prime}(w_1)=\pi^{\prime}(w_2)$, we see that $x^{j-i}(w^{\prime})^{-1}_1w^{\prime}_2\in K$, and so $ht(x^{j-i})+ht(w^{\prime})^{-1}_1w^{\prime}_2)\equiv 0(\mathrm{mod}\ n-1)$. In the same time, we have $ht((w^{\prime})^{-1}_1w^{\prime}_2)\equiv 0(\mathrm{mod}\ n-1)$. This shows that $ht(x^{j-i})\equiv 0(\mathrm{mod}\ n-1)$. Hence, the number $(j-i)$ is divisible by $n-1$. But $0\leq j-i\leq n-2$, and hence $i=j$. Form $\pi^{\prime}(w_1)=\pi^{\prime}(w_2)$, we conclude that $\pi^{\prime}(w^{\prime}_1)=\pi^{\prime}(w^{\prime}_2)$. Therefore,
\begin{eqnarray*}
h_0(w_1)&=&h_0(x)^{i-1}h_0(w^{\prime}_1)\\
        &=&h_0(x)^{i-1}\beta(\pi(w^{\prime}_1))\\
        &=&h_0(x)^{i-1}\beta(\pi^{\prime}(w^{\prime}_1))\\
        &=&h_0(x)^{i-1}\beta(\pi^{\prime}(w^{\prime}_2))\\
        &=&h_0(w_2).
\end{eqnarray*}
This shows that $h$ is well-defined. It is now, easy to see that $h_{|_G}=\beta$ and so $M$ is the Post's cover of $(G, f)$.
\end{proof}

Using this theorem, one can determine the explicit form of the Post's cover for various polyadic groups. As a very simple example, consider the singleton $n$-ary group $(G, f)=\langle x|\overline{x}\approx x\rangle$. By the above theorem, $G^{\ast}=\langle x| x^{n-1}\approx 1\rangle=\mathbb{Z}_{n-1}$. As another example, consider the polyadic group $(G, f)=\langle x, y|\overline{x}\approx x\rangle$. Then, the Post's cover is
$$
G^{\ast}=\langle x, y| x^{n-1}\approx 1\rangle=\mathbb{Z}\ast \mathbb{Z}_{n-1}.
$$
As final example, consider the polyadic group
$$
(G, f)=\langle x, y| f(x, y, \stackrel{(n-2)}{x})\approx f(y, \stackrel{(n-1)}{x})\rangle.
$$
Its Post's cover is
$$
G^{\ast}=\langle x, y|xy\approx yx\rangle= \mathbb{Z}\times \mathbb{Z}.
$$

\section{Algebraic geometry over polyadic groups}
We are now ready to study equations in polyadic groups. We assume that our $n$-ary group has the form $(G, f)=\Gf$. Let $X=\{ x_1, \ldots, x_m\}$ be a set of variables and define $G[X]=G\ast F_{pol}^n(X)$, where $\ast$ denotes the free product of polyadic groups. Elements of $G[X]$ are polyadic combinations of the elements of $X$ as well as elements from $G$. We can use a similar argument as in the construction of the free polyadic group $F_{pol}^n(X)$ to investigate the structure of $G[X]$. In fact, elements of $G[X]$ have the form
$$
a_1x_{j_1}^{\epsilon_1}a_2x_{j_2}^{\epsilon_2}\cdots a_kx_{j_k}^{\epsilon_k},
$$
where $k\geq 1$, $a_1, \ldots, a_k\in G$, $x_{j_1}, \ldots, x_{j_k}\in X$, and $\sum_i\epsilon_i\equiv 1(\mathrm{mod}\ n-1)$. If we denote the $n$-ary operation of $G[X]$ by $\delta$, then we have
$$
\delta(w_1, \ldots, w_n)=w_1w_2\cdots w_n.
$$
In the special case when all $w_i\in G$ we have
$$
\delta(w_1, \ldots, w_n)=f(w_1, \ldots, w_n).
$$
As is the previous section, we can prove that the Post's cover of $G[X]$ is the ordinary free product $G^{\ast}\ast F(X)$, where $G^{\ast}$ is the Post's cover of $(G, f)$.

Any element $p$ in this group depends on the variables $x_1, \ldots, x_m$, therefore,  we can denote it by $p=p(x_1, \ldots, x_m)$. We call such an element a {\em polyadic term}.  A {\em polyadic equation with coefficients from $G$} is a formal expression
$$
p(x_1, \ldots, x_m)\approx q(x_1, \ldots, x_m),
$$
where both $p$ and $q$ are polyadic terms with coefficients from $G$. An element $(g_1, \ldots, g_m)\in G^m$ is a {\em solution} of this equation, if the equality
$$
p(g_1, \ldots, g_m)= q(g_1, \ldots, g_m)
$$
is true in $G$. A {\em system of equations} is a set $S$ of polyadic equations. If $S$ is an arbitrary system of equations, we call its solution set
$$
V_G(S)=\{(g_1, \ldots, g_m)\in G^m: \forall (p\approx q)\in S\ \   p(g_1, \ldots, g_m)=q(g_1, \ldots, g_m)\}
$$
the corresponding {\em algebraic set}. For an arbitrary subset $Y\subseteq G^m$, its {\em radical} is a congruence of $G[X]$ consisting of all pairs
$(p, q)\in G[X]^2$ such that for all $(g_1, \ldots, g_m)\in Y$ we have $p(g_1, \ldots, g_m)=q(g_1, \ldots, g_m)$. We denote this radical by $\mathrm{Rad}(Y)$. The quotient polyadic group $\Gamma(Y)=G[X]/\mathrm{Rad}(Y)$ is called the {\em coordinate polyadic group} of $Y$. We say that $(G, f)$ is {\em equational noetherian}, if for any system $S$ of polyadic equations, there exists a finite subsystem $S_0$ such that $V_G(S)=V_G(S_0)$. For general properties of equational noetherian algebraic systems, see \cite{DMR2}.

One can define a topology on $G^m$ using algebraic sets as a prebase: in this topology every closed set is an arbitrary intersection of finite unions of algebraic sets. It is called the {\em Zariski topology} on $G^m$. An algebraic set is called {\em irreducible}, if it can not be covered non-trivially by two other closed sets. Two algebraic sets $Y_1$ and $Y_2$ are isomorphic, if there is a polyadic term bijection between them whose inverse is also a polyadic term map. The classification of algebraic sets up to isomorphism, determines the geometric properties of $(G, f)$. It can be shown that $Y_1$ and $Y_2$ are isomorphic, if and only if,  they have the same coordinate polyadic groups. Therefore, it is very important to determine the structure of polyadic groups which are the coordinate polyadic groups of algebraic sets in $G$. The {\em unification theorems} of \cite{DMR2} provide strong tools to do this. Here, we give two versions of the unification theorems which we can use to determine coordinate polyadic groups. For details, see \cite{DMR2}.

\begin{theorem}
Let $(G, f)$ be an equational noetherian polyadic group. For a finitely generated polyadic group $H$, the following conditions are equivalent.\\

1- $H$ embeds in a direct power of $(G, f)$.\\

2- $H$ satisfies all quasi-identities which are true in $(G, f)$.\\

3- $H$ is residually $(G, f)$-polyadic group.\\

4- $H$ is a coordinate polyadic group of some algebraic set in $(G, f)$.
\end{theorem}

Hence, applying this theorem, we can determine the class of all finitely generated polyadic groups which are the coordinate polyadic groups of algebraic sets in $G$. Note that, if $I$ is an arbitrary index set then
$$
(G, f)^I=\mathrm{der}_{\hat{\theta}, (b)}(G^I, \cdot),
$$
where $\hat{\theta}:G^I\to G^I$ is the automorphism induced by $\theta$,
$$
\hat{\theta}((g_i)_{i\in I})=(\theta(g_i))_{i\in I},
$$
and $(b)$ is the constant $I$-sequence with all terms equal to $b$. So, by the above unification theorem, we must determine all finitely generated polyadic subgroups of $\mathrm{der}_{\hat{\theta}, (b)}(G^I, \cdot)$. For this purpose, we need some facts from our earlier work \cite{Khod-Shah2}. For an arbitrary polyadic group $(G, f)=\Gf$, we  determine its all polyadic subgroups as follows.
For $u\in G$, define a new binary operation on $G$ by $x\ast y=xu^{-1}y$. Then $(G, \ast)$ is an isomorphic copy of $(G,\cdot)$ and the isomorphism is the map $x\mapsto xu$. We denote this new group by $G_u$. Its identity is $u$ and the inverse of $x$ is $ux^{-1}u$. We define an automorphism of $G_u$ by $\psi_u(x)=u\theta(x)\theta(u^{-1})$. It can be easily checked that this is actually an automorphism of $G_u$.

Suppose $H\leq (G, f)$. We denote the restriction of $f$ to $H$ by $f$, so there is a binary operation on $H$, say $\ast$, an automorphism $\psi$ and an element $c\in H$, such that
$$
(H, f)=\mathrm{der}_{\psi, c}(H, \ast).
$$
The inclusion map $j:H\to G$ is a polyadic homomorphism, hence by \cite{Khod-Shah}, there is an element $u\in G$ and an ordinary homomorphism $\phi:(H, \ast)\to (G,\cdot)$, with the properties\\

{\em 1- $j=R_u\phi$,\\

2-  $f(\stackrel{(n)}{u})=\phi(c)u$,\\

3- $\phi\psi=I_u\theta\phi$}.\\

Here $I_u$ is the inner automorphism corresponding to $u$. From 1, we deduce that for any $x\in H$, $\phi(x)=xu^{-1}$, and so by 2, we have $f(\stackrel{(n)}{u})=c$. Moreover, since $\phi$ is an ordinary homomorphism, so using $\phi(x\ast y)=\phi(x)\phi(y)$ and $\phi(x)=xu^{-1}$, we obtain $x\ast y=xu^{-1}y$. Finally, by 3, we have $\psi(x)=u\theta(x)\theta(u^{-1})$, and therefore we must have $(H, \ast)\leq G_u$. Further, $H$ is invariant under $\psi_u$ and hence $\psi=\psi_u|_H$. So, we proved that  $H$ is a polyadic subgroup of $(G, f)$ iff there exists an element $u$ such that $H$ is a $\psi_u$-invariant subgroup of $G_u$.

Having determined the structure of polyadic subgroup, we now apply the first unification theorem, to obtain the following result. The proof is now trivial, because we just determine the polyadic subgroups of $\mathrm{der}_{\hat{\theta}, (b)}(G^I, \cdot)$

\begin{theorem}
Let $H$ be a coordinate polyadic group of some algebraic set in $(G, f)$. Let $(G, f)$ be equational noetherian. Then there exist a set $I$ and an element $u=(u_i)_{i\in I}\in G^I$ such that\\

1- $H$ is an ordinary subgroup of $(G^I)_u$.\\

2- For all $(h_i)_{i\in I}\in H$, we have $(u_i\theta(h_i)\theta(u_i^{-1}))_{i\in I}\in H$.\\

The converse is also true for all finitely generated polyadic groups having properties 1 and 2.
\end{theorem}

Note that, by the above unification theorem,  finitely generated polyadic groups described by 1 and 2 in this theorem are the only finitely generated elements of the quasi-variety generated by $(G, f)$. They are also the only finitely generated residually $(G, f)$-polyadic groups.  The second unification theorem gives a description of the coordinate polyadic groups of irreducible algebraic sets.

\begin{theorem}
Let $(G, f)$ be an equational noetherian polyadic group. For a finitely generated polyadic group $H$, the following conditions are equivalent.\\

1- $H$ embeds in an ultra-power of $(G, f)$.\\

2- $H$ satisfies all universal sentences  which are true in $(G, f)$.\\

3- $H$ is fully residually $(G, f)$-polyadic group.\\

4- $H$ is a coordinate polyadic group of some irreducible algebraic set in $(G, f)$.
\end{theorem}

Let $I$ be a set and $F$ be an ultra-filter over $I$. For the polyadic group $(G, f)=\Gf$, one can easily checks that
$$
\frac{(G, f)^I}{F}=\mathrm{der}_{\theta^{\ast}, (b)/F}(\frac{G^I}{F}, \cdot),
$$
where $\theta^{\ast}:G^I/F\to G^I/F$ is the automorphism
$$
\theta^{\ast}(\frac{(g_i)_{i\in I}}{F})=\frac{(\theta(g_i))_{i\in I}}{F},
$$
and $(b)/F$ is the equivalence class of the $I$-sequence $(b)$. Again, we can apply 4.2 to determine the structure of coordinate polyadic groups of irreducible algebraic sets over $(G, f)$.

\begin{theorem}
Let $H$ be a coordinate polyadic group of some irreducible algebraic set in $(G, f)$. Let $(G, f)$ be equational noetherian. Then there exist a set $I$, an ultra-filter $F$ over $I$ and an element $u=(u_i)/F\in G^I/F$ such that\\

1- $H$ is an ordinary subgroup of $(G^I/F)_u$.\\

2- For all $(h_i)/F\in H$, we have $(u_i\theta(h_i)\theta(u_i^{-1}))/F\in H$.\\

The converse is also true for all finitely generated polyadic groups having properties 1 and 2.
\end{theorem}

Again, we see that finitely generated polyadic groups described by 1 and 2 in this theorem are the only finitely generated elements of the universal class generated by $(G, f)$. They are also the only finitely generated fully residually $(G, f)$-polyadic groups.

\section{Equational noetherian polyadic groups}
In this section, we show that there is a close connection between algebraic geometries  of the polyadic group $(G, f)=\Gf$ and the ordinary group $(G, \cdot)$.

Let $\mathcal{L}$ be the language of groups, consisting of one binary symbol for product, one unary symbol for inverse and all elements of $G$ as constants, $\mathcal{L}_{\theta}$ be the same language but having $\theta$ as the second  unary operation. Assume that also $\mathcal{L}_p^n$ is the language of polyadic groups having one $n$-ary operation $f$, one unary operation $^-$, and elements of $G$ as constants.

\begin{theorem}
The polyadic group $(G, f)$ is equational noetherian, if and only if, the Post's cover $(G^{\ast}, \circ)$ is equational noetherian, if and only if $(G, \cdot)$ is equational noetherian.
\end{theorem}

\begin{proof}
For arbitrary elements $a_1, \ldots, a_n\in G$, we have inside $(G^{\ast}, \circ)$,
$$
f(a_1, \ldots, a_n)=a_1\circ a_2\circ\cdots\circ a_n,
$$
so, every polyadic equation with coefficients from $G$ is in the same time a group equation with coefficients from $G^{\ast}$. Therefore, if $G^{\ast}$ is equational noetherian, then so is the polyadic group $(G, f)$. Also, since $(G, \cdot)$ is a subgroup of $G^{\ast}$, so $(G, \cdot)$ is equational noetherian. Now, suppose $(G, f)$ is equational noetherian. First, we show that $(G, \cdot)$ is equational noetherian. We know that there exists an element $a\in G$ such that $(G, \cdot)=ret_a(G, f)$, hence for arbitrary elements $u, v\in G$ we have $u\cdot v=f(u,\stackrel{(n-2)}{a}, v)$. Also, we have $u^{-1}=f(\overline{a}, \stackrel{(n-3)}{u}, \overline{u}, \overline{a})$. This observation shows that every group equation in the language $\mathcal{L}$ can be transformed into an equation in the language $\mathcal{L}_p^n$. For example, consider the ordinary equation $bx^2y^{-1}cx\approx 1$. We have
\begin{eqnarray*}
1&\approx& f(b,\stackrel{(n-2)}{a}, x^2y^{-1}x)\\
 &\approx& f(b,\stackrel{(n-2)}{a}, f(x, \stackrel{(n-2)}{a}, xy^{-1}x))\\
 &\approx& f(b,\stackrel{(n-2)}{a}, f(x, \stackrel{(n-2)}{a}, f(x, \stackrel{(n-2)}{a},y^{-1}x)))\\
 &\approx& f(b,\stackrel{(n-2)}{a}, f(x, \stackrel{(n-2)}{a}, f(x, \stackrel{(n-2)}{a}, f(y^{-1},\stackrel{(n-2)}{a},  x))))\\
 &\approx& f(b,\stackrel{(n-2)}{a}, f(x, \stackrel{(n-2)}{a}, f(x, \stackrel{(n-2)}{a}, f(f(\overline{a}, \stackrel{(n-3)}{y}, \overline{y}, \overline{a}),\stackrel{(n-2)}{a},  x)))).
\end{eqnarray*}
This argument shows that $(G, \cdot)$ is equational noetherian. It is proved in \cite{BMRom} that every finite extension of an equational noetherian group is also equational noetherian. Since $G^{\ast}$ is a finite extension of $(G, \cdot)$, so it is also equational noetherian.

\end{proof}

Note that by the same method, we can prove  similar statements for the properties of being $q_{\omega}$-compact and $u_{\omega}$-compact. This shows that the algebraic geometric properties of groups $(G, \cdot)$, $(G^{\ast}, \circ)$ and the polyadic group $(G, f)$ are the same. As an interesting result of the above theorem, we see that every free polyadic group is equational noetherian.  Note that in our argument, the element $b$ and the automorphism $\theta$ have no role. So, we have the following corollary.

\begin{corollary}
Let $G$ be a group having an automorphism $\theta$ with the following property: there exists an element $b\in G$ and a number $n>2$ such that $\theta(b)=b$ and $\theta^{n-1}(x)=bxb^{-1}$, for all $x$. Then $G$ is equational noetherian in the language $\mathcal{L}$, if and only if, it is equational noetherian in the language $\mathcal{L}_{\theta}$.
\end{corollary}

As we saw, there is a close connection between algebraic geometries of $(G, f)$ and $G^{\ast}$. Now, we obtain a relation between coordinate polyadic group of a system of equations over $(G, f)$ and the coordinate group of the same system over $G^{\ast}$. For the sake of simplicity, we work in the case of coefficient-free equations.

Let $X=\{ x_1, \ldots, x_m\}$. If $S\subseteq F_{pol}^n(X)$ is a system of polyadic equations, then in the same time we can consider $S$ as a system of ordinary group equations. Recall that the radical of the system $S$ in $(G, f)$ and $G^{\ast}$ is defined by
$$
\mathrm{Rad}_G(S)=\mathrm{Rad}(V_G(S)) \ \ \mathrm{and}\ \   \mathrm{Rad}_{G^{\ast}}(S)=\mathrm{Rad}(V_{G^{\ast}}(S))
$$
respectively. The first one is a congruence in $F_{pol}^n(X)$ and the second one is a normal subgroup of $F(X)$. The corresponding coordinate {\em algebras} are
$$
\Gamma_G(S)=\frac{F_{pol}^n(X)}{\mathrm{Rad}_G(S)} \ \ \mathrm{and} \ \ \Gamma_{G^{\ast}}(S)=\frac{F(X)}{\mathrm{Rad}_{G^{\ast}}(S)}.
$$

\begin{theorem}
The coordinate group $\Gamma_{G^{\ast}}(S)$ is a quotient of the Post's cover of $\Gamma_G(S)$.
\end{theorem}

\begin{proof}
Note that $\mathrm{Rad}_G(S)$ is the set of all polyadic consequences of the elements of $S$ in $(G, f)$, while $\mathrm{Rad}_{G^{\ast}}(S)$ is the set of all ordinary consequences of the elements of $S$ in $G^{\ast}$. Since $G\subseteq G^{\ast}$ and all polyadic equations are in the same time group equations, so we have $\mathrm{Rad}_G(S)\subseteq \mathrm{Rad}_{G^{\ast}}(S)$. Recall that
$$
\Gamma_{G^{\ast}}(S)=\langle X| \mathrm{Rad}_{G^{\ast}}(S)\rangle_{gr}.
$$
Hence $\Gamma_{G^{\ast}}(S)$ is a quotient of the group $\langle X| \mathrm{Rad}_G(S)\rangle_{gr}$. By theorem 4.1, the recent group is the Post's cover of $\Gamma_G(S)=\langle X| \mathrm{Rad}_G(S)\rangle_{pol}$. This completes the proof.
\end{proof}

\end{document}